\documentclass[reqno,12pt]{amsart}

\usepackage{amsmath}
\usepackage{amsfonts}
\usepackage{amssymb}
\usepackage{amsthm}

\usepackage[mathscr]{eucal}\catcode`\@=11

\theoremstyle{plain}
\newtheorem{theorem}{Theorem}
\newtheorem{prop}{Proposition}
\newtheorem{lemma}[prop]{Lemma}

\begin{document}

\subjclass[2010]{11D41, 11B68}
\keywords{Diophantine equations, Bernoulli polynomials, Euler polynomials}

\title[Effective results for polynomial values of power sums]{Effective results for polynomial values of (alternating) power sums of arithmetic progressions}

\author[Andr\'as Bazs\'o]{Andr\'as Bazs\'o}

\address{Andr\'as Bazs\'o \newline
\indent Institute of Mathematics \newline
\indent University of Debrecen \newline
\indent P.O. Box 400, H-4002 Debrecen, Hungary \newline
\indent and \newline
\indent HUN-REN-UD Equations, Functions, Curves and their Applications Research Group}
\email{bazsoa@science.unideb.hu}


\begin{abstract}
We prove effective finiteness results concerning polynomial values of the sums 
$$
b^k  +\left(a+b\right)^k + \cdots + \left(a\left(x-1\right) + b\right)^k
$$
and
$$
b^k - \left(a+b\right)^k + \left(2a+b\right)^k - \ldots + (-1)^{x-1} \left(a\left(x-1\right) + b\right)^k ,
$$
where $a \neq 0,b, k$ are given integers with $\gcd(a,b)=1$ and $k \geq 2$. 
\end{abstract}

\maketitle

\section{Introduction}

Problems concerning power sums and alternating power sums of consecutive integers has a long history in the literature of combinatorics and number theory. It is well known that the sum
\begin{equation}
S_k (n) = 1^k + 2^k + \ldots + (n-1)^k
\end{equation}
can be expressed by the Bernoulli polynomials $B_k(x)$ as
\begin{equation}
S_k (n) = \frac{1}{k+1} \left(B_{k+1} (n) - B_{k+1}\right),
\end{equation}
where the polynomials $B_k(x)$ are defined by the generating series 
$$
\frac{te^{tx}}{e^t-1}=\sum_{m=0}^{\infty}B_{m}(x)\frac{t^m}{m!}
$$
and $B_{k+1} = B_{k+1} (0)$. Hence $S_k$ can be extended to real values $x$, i.e., to the polynomial 
\begin{equation}
\texttt{S}_k (x) = \frac{1}{k+1} \left(B_{k+1} (x) - B_{k+1}\right).
\end{equation}

It is also well known that the alternating power sum
$$
T_k \left(n\right) := -1^k +2^k - \cdots + (-1)^{n-1} (n-1)^k
$$
can be expressed by means of the classical Euler polynomials $E_k (x)$ via:
$$
T_k \left(n\right) = \frac{E_k (0) + (-1)^{n-1} E_k (n)}{2},
$$
where the classical Euler polynomials $E_k (x)$ are usually defined by the generating function
$$
\frac{2 e^{xt}}{e^t + 1} = \sum^{\infty}_{m=0}{E_m (x) \frac{t^m}{m!}} \ \ \ (|t| < \pi).
$$

For the properties of Bernoulli and Euler polynomials which will be often used in this paper, sometimes without special reference, we refer to the paper of Brillhart \cite{Brill} and the book of Abramowitz and Stegun \cite{AbrSt}.

A classical problem of Lucas \cite{L}, from 1875, was the study of square values of $\texttt{S}_k (x)$. Later, in 1956, Sch\"affer \cite{schaf} investigated $n$-th power values, that is, the Diophantine equation  
\begin{equation} \label{eq:Sch}
\texttt{S}_k (x) = y^n   \ \ \ \ \text{in integers $x,y$.}
\end{equation}
For $k \geq 1$, $n \geq 2$ he proved an ineffective finiteness result on the solutions $x,y$ of \eqref{eq:Sch} provided that $(k,n) \notin \left\{(1,2), (3,2), (3,4), (5, 2)\right\}$. In the exceptional cases $(k,n)$ he proved the existence of infinitely many solutions. Moreover, Sch\"affer proposed a still unproven conjecture which says that if $(k,n)$ is not in the above exceptional set, then the only nontrivial solution of equation \eqref{eq:Sch} is $(k, n, x, y) = (2, 2, 24, 70)$. In 1980, Gy\H{o}ry, Tijdeman and Voorhoeve \cite{GyTV} proved effective finiteness for the solutions of \eqref{eq:Sch} in the general case when, in \eqref{eq:Sch}, $n$ is also unknown. Several generalizations of \eqref{eq:Sch} have been considered, e.g. in the papers of Voorhoeve, Gy\H{o}ry and Tijdeman \cite{VGyT}, Brindza \cite{Brind}, Dilcher \cite{Dilch} and Urbanowicz \cite{Urb1,Urb2,Urb3}.
Sch\"affer's conjecture has been confirmed only in a few cases: for $n=2$ and $k \leq 58$ by Jacobson, Pint\'er and Walsh \cite{JPW}; and for $n \geq 2$ and $k \leq 11$ by Bennett, Gy\H{o}ry and Pint\'er \cite{BGyP}. For further generalizations of \eqref{eq:Sch} and related results see the survey paper of Gy\H{o}ry and Pint\'er \cite{GyP} and the references given there.

In \cite{BBKPT}, Bilu et al. considered the Diophantine equations
\begin{equation} \label{eq:Skx}
\texttt{S}_k (x) = \texttt{S}_{\ell} (y),
\end{equation}
and
\begin{equation} \label{eq:SkRl}
\texttt{S}_k (x) = y\left(y+1\right) \left(y+2\right) \ldots \left(y+ \left(\ell-1\right)\right).
\end{equation}
They proved ineffective finiteness results on the integer solutions $x,y$ of these equations for $k<\ell$, moreover, they established effective statements for certain small values of $k$ and $\ell$. E.g., they proved that, for $\ell \in \left\{2,4\right\}$ and $k\geq 2$, all integer solutions $x,y$ of equation \eqref{eq:SkRl} satisfy $\max \left\{|x|,|y|\right\} < C(k)$, where $C(k)$ is an effectively computable constant depending only on $k$.

For a positive integer $n\geq 2$ and for $a \neq 0, b$ coprime integers, let
\begin{equation}
S_{a,b}^k \left(n\right) = b^k + \left(a+b\right)^k + \left(2a+b\right)^k + \ldots + \left(a\left(n-1\right) + b\right)^k .
\end{equation}
Further, let
\begin{equation}
T_{a,b}^k \left(n\right) = b^k - \left(a+b\right)^k + \left(2a+b\right)^k - \ldots + (-1)^{n-1} \left(a\left(n-1\right) + b\right)^k.
\end{equation}

Howard \cite{How} showed, by means of generating functions, that the above power sums are related to the Bernoulli and Euler polynomials in the following ways
\begin{equation} 
S_{a,b}^k \left(n\right) = \frac{a^k}{k+1} \left(B_{k+1} \left(n+ \frac{b}{a}\right) - B_{k+1} \left(\frac{b}{a}\right)\right),
\end{equation}
\begin{equation} \label{eq:Tkabn} 
T_{a,b}^k \left(n\right) = \frac{a^k}{2} \left(E_k \left(\frac{b}{a}\right) + (-1)^{n-1} E_k \left(n + \frac{b}{a}\right)\right),
\end{equation}
respectively.

Thus we can extend $S_{a,b}^k$ for every real value $x$ as
\begin{equation} \label{eq:How1}
\texttt{S}_{a,b}^k \left(x\right) =  \frac{a^k}{k+1} \left(B_{k+1} \left(x+ \frac{b}{a}\right) - B_{k+1} \left(\frac{b}{a}\right)\right),
\end{equation}
and, depending on the power of $-1$ in \eqref{eq:Tkabn}, the following polynomial extensions arise for $T_{a,b}^k \left(n\right)$:
\begin{equation} \label{eq:How2}
\texttt{T}_{a,b}^{k+} (x) = \frac{a^k}{2} \left(E_k \left(\frac{b}{a}\right) + E_k \left(x + \frac{b}{a}\right)\right),
\end{equation}
\begin{equation} \label{eq:How3}
\texttt{T}_{a,b}^{k-} (x) = \frac{a^k}{2} \left(E_k \left(\frac{b}{a}\right) - E_k \left(x + \frac{b}{a}\right)\right).
\end{equation}
Clearly, for positive integer values $x$, we have $\texttt{T}_{a,b}^{k+} (x) = T_{a,b}^k \left(x\right)$ if $x$ is odd, and $\texttt{T}_{a,b}^{k-} (x) = T_{a,b}^k \left(x\right)$ if $x$ is even. 

In \cite{BKLP}, Kreso, Luca, Pint\'er and the present author generalized the results of Bilu et al. \cite{BBKPT} on equation \eqref{eq:Skx} to the equation
\begin{equation} \label{eq:Skabx}
\texttt{S}_{a,b}^k \left(x\right) = \texttt{S}_{c,d}^{\ell} \left(y\right)
\end{equation}
where $x,y$ are unknown integers, and $k,\ell,a,b,c,d$ are given positive integers with $k<\ell, \ \gcd (a,b) = \gcd (c,d) = 1$.
Recently, in \cite{BKLPR}, Kreso, Luca, Pint\'er, Rakaczki and the author gave further generalizations of the results of  \cite{BBKPT} to the equation 
\begin{equation} \label{eq:RS}
\texttt{S}_{a,b}^k (x) = y\left(y+c\right) \left(y+2c\right) \ldots \left(y+ \left(\ell-1\right)c\right)
\end{equation}
where $x,y$ are unknown integers, and $k,\ell,a,b,c$ are given positive integers with $\gcd (a,b) =1$.

The Diophantine equations
\begin{equation} \label{eq:Sg}
\texttt{S}_{a,b}^k \left(x\right) = g(y),
\end{equation}
\begin{equation} \label{eq:Tg+}
\texttt{T}_{a,b}^{k+} \left(x\right) = g(y)
\end{equation}
and
\begin{equation} \label{eq:Tg-}
\texttt{T}_{a,b}^{k-} \left(x\right) = g(y),
\end{equation}
where $g(y) \in \mathbb{Q}[y]$ with $\deg g \geq 3$ have been investigated in the literature first in the case $(a,b)=(1,0)$. Rakaczki \cite{R} and independently Kulkarni and Sury \cite{KS} characterized those pairs $(k,g(y))$ for which equation \eqref{eq:Sg} has infinitely many integer solutions. Kreso and Rakaczki \cite{KR} proved an analogous result for equations \eqref{eq:Tg+} and \eqref{eq:Tg-}. For further related results we refer to the papers of Kulkarni and Sury \cite{KS1,KS2}, and of Bennett \cite{Ben13}.

The results of \cite{R,KS,KR} have been extended to the general case, i.e, to equations \eqref{eq:Sg}-\eqref{eq:Tg-} by the present author in \cite{BA2}. We note that all the mentioned results on equations \eqref{eq:Sg}-\eqref{eq:Tg-} were ineffective as their proofs mainly relied on the finiteness criterion of Bilu and Tichy \cite{BiluTichy}, and on the decomposition properties of the polynomials involved (which were described in the papers \cite{BA,BPS,BBKPT}). For a detailed discussion on the theory of polynomial decomposition we refer to the monograph of Schinzel \cite{Sch}.

In the present paper we consider equations \eqref{eq:Sg}-\eqref{eq:Tg-} in the cases when $g(y) \in \mathbb{Q}[y]$ with $\deg g = 2$, and for $g(y)=cy^{\ell}+d$ with $c,d \in \mathbb{Q}$, $c\neq 0$ and $\ell \geq 2$. In the quadratic case we provide effective upper bounds on the solutions $(x,y)$, while in the latter case we also let $\ell$ be unknown, and prove effective finiteness for the triple $(x,y,\ell)$ as well. Our results have important special cases. They imply, e.g., that there are only finitely many power values of $\texttt{S}_{a,b}^k \left(x\right)$ which can be considered as a generalization of the results of \cite{GyTV}. Further, our Theorems \ref{thm:quadpolS}, \ref{thm:quadpolT+}, \ref{thm:quadpolT-} extend the main results of \cite{BA2} in an effective way to the quadratic case. 

\section{Main results} \label{sec:mainres}

Let $a \neq 0, b,k \in \mathbb{Z}, \gcd(a,b)=1$, and let $g \in \mathbb{Q}[x]$ be an arbitrary polynomial. Consider the Diophantine equations
\begin{equation} \tag{14}
\texttt{S}_{a,b}^k \left(x\right) = g(y),
\end{equation}
\begin{equation} \tag{15}
\texttt{T}_{a,b}^{k+} \left(x\right) = g(y)
\end{equation}
and
\begin{equation} \tag{16}
\texttt{T}_{a,b}^{k-} \left(x\right) = g(y),
\end{equation}
in $x,y \in \mathbb{Z}$.

First we consider the case when $g$ is a quadratic polynomial. On equation \eqref{eq:Sg} we prove the following.

\begin{theorem} \label{thm:quadpolS}
Let $a,b,k$ and $g$ be as above. Then for $k \geq 2$, $k \notin \left\{3,5\right\}$, and $\deg g = 2$, there exists an effectively computable constant $C_1(a,b,k,g)$ depending only on $a,b,k$ and $g$ such that $\max(|x|,|y|) < C_1(a,b,k,g)$ for each integer solutions of equation \eqref{eq:Sg}.
\end{theorem}

As we mentioned above, this result extends Theorem 1 of \cite{BA2} in an effective way.

In the alternating case, i.e., for equations \eqref{eq:Tg+},\eqref{eq:Tg-}, we have

\begin{theorem} \label{thm:quadpolT+}
Let $a,b,k$ and $g$ be as above. Then for $k \geq 7$ and $\deg g = 2$, there exists an effectively computable constant $C_2(a,b,k,g)$ depending only on $a,b,k$ and $g$ such that $\max(|x|,|y|) < C_2(a,b,k,g)$ for each integer solutions of equation \eqref{eq:Tg+}.
\end{theorem}

\begin{theorem} \label{thm:quadpolT-}
Let $a,b,k$ and $g$ be as above. Then for $k \geq 7$ and $\deg g = 2$, there exists an effectively computable constant $C_3(a,b,k,g)$ depending only on $a,b,k$ and $g$ such that $\max(|x|,|y|) < C_3(a,b,k,g)$ for each integer solutions of equation \eqref{eq:Tg-}.
\end{theorem}

Analogous to Theorem \ref{thm:quadpolS}, the above two theorems extend Theorem 2 and 3 of \cite{BA2}, respectively.

Now we turn our attention to the case when, in equations \eqref{eq:Sg}-\eqref{eq:Tg-}, the polynomial $g$ is linear in some power of $y$. More precisely, we consider the equations
\begin{equation} \label{eq:Sling}
\texttt{S}_{a,b}^k \left(x\right) = cy^{\ell}+d,
\end{equation}
\begin{equation} \label{eq:Tling+}
\texttt{T}_{a,b}^{k+} \left(x\right) = cy^{\ell}+d
\end{equation}
and
\begin{equation} \label{eq:Tling-}
\texttt{T}_{a,b}^{k-} \left(x\right) = cy^{\ell}+d,
\end{equation}
in integers $x,y$ and $\ell \geq 2$, where $c,d \in \mathbb{Q}$ with $c\neq 0$.

In this case we can also bound the exponent $\ell$ from above together with $x$ and $y$.

\begin{theorem} \label{thm:linpolS}
Let $a,b,c,d,k$ and $\ell$ be as above. Then for $k \geq 2$, $k \notin \left\{3,5\right\}$, there exists an effectively computable constant $C_4(a,b,c,d,k)$ depending only on $a,b,c,d$ and $k$ such that $\max(|x|,|y|,\ell) < C_4(a,b,c,d,k)$ for each integer solutions of equation \eqref{eq:Sling} with $|y|>1$. 
\end{theorem}

In the special case $c=1, d=0$, our Theorem \ref{thm:linpolS} implies effective finiteness for the power values of $\texttt{S}_{a,b}^k \left(x\right)$, i.e. for a generalization of Sch\"affer's equation \eqref{eq:Sch}. It can be considered as an extension of Theorem 1 of \cite{GyTV} as well. We note that, in the above mentioned paper \cite{BKLPR}, we also proved effective finiteness for the power values of $\texttt{S}_{a,b}^k \left(x\right)$ in a slightly more general case $k\geq 1, \ell \geq 2$ (cf. Theorem 1.1 in \cite{BKLPR}). The reader can find a detailed analysis of the exceptional cases $k \in \left\{1,3,5\right\}$ there.

We prove analogous results in the alternating case. 

\begin{theorem} \label{thm:linpolT+}
Let $a,b,c,d,k$ and $\ell$ be as above. Then for $k \geq 7$, there exists an effectively computable constant $C_5(a,b,c,d,k)$ depending only on $a,b,c,d$ and $k$ such that $\max(|x|,|y|,\ell) < C_5(a,b,c,d,k)$ for each integer solutions of equation \eqref{eq:Tling+} with $|y|>1$. 
\end{theorem}

\begin{theorem} \label{thm:linpolT-}
Let $a,b,c,d,k$ and $\ell$ be as above. Then for $k \geq 7$, there exists an effectively computable constant $C_6(a,b,c,d,k)$ depending only on $a,b,c,d$ and $k$ such that $\max(|x|,|y|,\ell) < C_6(a,b,c,d,k)$ for each integer solutions of equation \eqref{eq:Tling-} with $|y|>1$.
\end{theorem}

\section{Proofs of the Theorems} \label{sec:proofs}

In this section, we give the proofs of our theorems.
In our arguments we need some lemmas and notation. First we recall two classical effective results on hyper- and superelliptic equations. 

Let $f(x) \in \mathbb{Z} [x]$ be a nonzero polynomial of degree $d$. Write $H$ for the naive height (i.e., the maximum of the absolute values of the coefficients) of $f$. Further, let $\alpha$ be a nonzero rational number. Consider the Diophantine equation
\begin{equation} \label{eq:hypsup}
f(x)= \alpha y^N.
\end{equation}

The following result is a special case of a result of B\'erczes, Brindza and Hajdu \cite{BBrH}. For the first results of this type, we refer to Schinzel and Tijdeman \cite{SchT} and Tijdeman \cite{T}.

\begin{lemma} \label{lem:BBrH}
If $f(x)$ has at least two distinct roots and $|y|>1$, then, in \eqref{eq:hypsup}, we have $N<C_7(d,H,\alpha)$, where $C_7(d,H,\alpha)$ is an effectively computable constant depending only on $d,H$ and $\alpha$.
\end{lemma}

The next result is a special case of an effective theorem of Brindza \cite{brindza}. To formulate it, we need further notation. For a finite set of rational primes $S$, let $\mathbb{Z}_S$ denote the set of rational numbers whose denominator (in reduced form) has no prime divisor outside $S$. By the \textit{height} $h(s)$ of a rational number $s=u/v$ with $u,v\in \mathbb{Z}$,  $\gcd(u,v)=1$, we mean $h(s)=\max \left\{|u|,|v|\right\}$.

\begin{lemma}\label{lem:Brindza}
Let $S$ be a finite set of rational primes. If, in \eqref{eq:hypsup}, either $N=2$ and $f(x)$ has at least three roots of odd multiplicity, or $N\geq 3$ and $f(x)$ has at least two roots of multiplicities coprime to $N$, then for each solutions $x,y \in \mathbb{Z}_S$ of \eqref{eq:hypsup} we have $\max(h(x),h(y))<C_8(d,H,S,\alpha,N)$, where $C_8(d,H,S,\alpha,N)$ is an effectively computable constant depending only on $d,H,S,\alpha$ and $N$.
\end{lemma}

We recall the following result concerning Bernoulli polynomials $B_k(x)$ which is due to Brillhart \cite{Brill}.

\begin{lemma} \label{lem:Br1}
If $k$ is odd, then $B_k(x)$ has no multiple roots. For even $k$, the only polynomial which can be a multiple factor of $B_k(x)$ over $\mathbb{Q}$ is $x^2-x-\beta$, where $\beta$ is an odd, positive integer. Further, the multiplicity of any multiple root of $B_k$ is $2$.
\end{lemma}

The following result, which provides information on the root structure of shifted Bernoulli polynomials, will be crucial in applying Lemma \ref{lem:Brindza} in the proofs of Theorems \ref{thm:quadpolS} and \ref{thm:linpolS}.

\begin{lemma} \label{lem:Bk+sroots}
For every $s\in \mathbb Q$ and rational integer $k\geq 3$ with $k\notin \{4,6\}$ the polynomial $B_k(x)+s$ has at least three roots of odd multiplicity.
\end{lemma}

\begin{proof}
This is Lemma 2.2 in \cite{BKLP}.
\end{proof}

Now we are in position to prove our Theorems \ref{thm:quadpolS} and \ref{thm:linpolS}.

\begin{proof}[Proof of Theorem \ref{thm:quadpolS}]
Let $k\geq 2$, $k\notin\left\{3,5\right\}$, and let $g \in \mathbb{Q}[x]$ be an arbitrary polynomial with  $\deg g = 2$. Assume that equation \eqref{eq:Sg} holds. Then there exist rational numbers $A,B,C$ with $A \neq 0$ such that
\begin{equation} \label{eq:effprf1}
g(x) = Ax^2+Bx+C.
\end{equation}
Obviously, we can rewrite \eqref{eq:Sg} as
\begin{equation} \label{eq:effprf2}
\texttt{S}_{a,b}^k \left(x\right) + \nu = A(y+\mu)^2,
\end{equation}
where $\mu=\frac{B}{2A}$ and $\nu=\frac{B^2-4AC}{4A}$. Thus, to bound $x,y$, in view of Lemma \ref{lem:Brindza}, it is sufficient to show that the polynomial $\texttt{S}_{a,b}^k \left(x\right) + \nu$ has at least three roots of odd multiplicity. By \eqref{eq:How1}, we have
\begin{equation} \label{eq:effprf3}
\texttt{S}_{a,b}^k \left(x\right) + \nu=\frac{a^k}{k+1} \left(B_{k+1} \left(x+ \frac{b}{a}\right) + s\right)
\end{equation}
with $s= - B_{k+1} \left(\frac{b}{a}\right) + \frac{(k+1)\nu}{a^k} \in \mathbb{Q}$,
and, since the number of roots as well as the multiplicities of the roots of a polynomial remain unchanged if we replace its variable by a linear polynomial of that, or if we multiply the polynomial by a nonzero constant, Lemma \ref{lem:Bk+sroots} implies the existence of three roots of odd multiplicity for $\texttt{S}_{a,b}^k \left(x\right) + \nu$. This completes the proof.
\end{proof}

\begin{proof}[Proof of Theorem \ref{thm:linpolS}]
We rewrite equation \eqref{eq:Sling} as
\begin{equation} \label{eq:effprf4}
\texttt{S}_{a,b}^k \left(x\right) - d = cy^{\ell}.
\end{equation}
To prove that $\ell$ is bounded, by Lemma \ref{lem:BBrH}, we only need to show that the polynomial $\texttt{S}_{a,b}^k \left(x\right) - d$ has two distinct roots for every $d\in \mathbb{Q}$. Assuming to the contrary, for some $d\in \mathbb{Q}$, we have
\begin{equation} \label{eq:effprf5}
\texttt{S}_{a,b}^k \left(x\right) - d = R(Ux+V)^{k+1}
\end{equation}
with some $R,U,V\in \mathbb{Q}$. However, as $k\geq 2$, this implies that the derivative
\begin{equation} \label{eq:effprf6}
(\texttt{S}_{a,b}^k \left(x\right) - d)'= \frac{a^k}{k+1} B_{k+1}' \left(x+ \frac{b}{a}\right) = a^k B_{k} \left(x+ \frac{b}{a}\right)
\end{equation}
has a rational root of multiplicity $k$. This contradicts Lemma \ref{lem:Br1}, whence $\ell$ is bounded as required. (Indeed, for $k\leq3$, then the Bernoulli polynomial $B_k$ would have a rational root of multiplicity at least $3$, and for $k=2$, $B_2$ would have a double rational root which are both impossible.)

Now we give a bound for $\max(|x|,|y|)$. For $\ell=2$, the statement is a straightforward consequence of Theorem \ref{thm:quadpolS}. In the sequel, let $\ell \geq 3$. By what we have already proved, we may assume that $\ell$ is fixed. Further, in view of Theorem \ref{thm:quadpolS}, we may suppose that $\ell$ is odd. Clearly, without loss of generality we may assume that in fact $\ell$ is an odd prime. 
By \eqref{eq:effprf4} and Lemma \ref{lem:Brindza}, it suffices to show that the polynomial on the left hand side of \eqref{eq:effprf4} has at least two roots of multiplicities coprime to $\ell$. Suppose to the contrary, i.e., that we have
\begin{equation} \label{eq:effprf7}
\texttt{S}_{a,b}^k \left(x\right) - d = (Kx+L)(w(x))^{\ell},
\end{equation}
with some $K,L \in \mathbb{Q}$, $w \in \mathbb{Q}[x]$. Taking derivatives in \eqref{eq:effprf7}, by \eqref{eq:effprf6}, we obtain
\begin{equation} \label{eq:effprf8}
a^k B_{k} \left(x+ \frac{b}{a}\right) = w(x)^{\ell-1} \left(K w(x) + \ell (Kx+L)w'(x)\right).
\end{equation}
Thus, every root of $w(x)$ is a root of $a^k B_{k} \left(x+ \frac{b}{a}\right)$ of multiplicity at least $\ell-1\geq 2$. This immediately contradicts Lemma \ref{lem:Br1} if $\ell \geq 4$.  
In the case $\ell=3$, Lemma \ref{lem:Br1} implies $w(x)=(x^2-x-\beta)$ with an odd positive integer $\beta$. Then, we obtain from \eqref{eq:effprf8} that $k=6$ and that $B_{6} \left(x+ \frac{b}{a}\right)$ should have a multiple root, which is a contradiction. 
This completes the proof.
\end{proof}

In the sequel, we turn to the alternating case, i.e., to the proofs of Theorems \ref{thm:quadpolT+},\ref{thm:quadpolT-} and \ref{thm:linpolT+},\ref{thm:linpolT-}. Before proving our results, besides the abovementioned auxiliary results, we need the following two lemmas. The first is a deep result of Rakaczki \cite{R2} concerning the root structure of shifted Euler polynomials.

\begin{lemma} \label{lem:R2}
Let $k \geq 7$ be an integer. Then the shifted Euler polynomial $E_k (x) + z$ has at least three simple zeros for arbitrary complex number $z$.
\end{lemma}

The following result, which describes the multiple roots of Euler polynomials $E_k(x)$ is due to Brillhart \cite{Brill}.

\begin{lemma} \label{lem:Br2}
If $k$ is even, then $E_k(x)$ has no multiple roots. For odd $k$, the only polynomial which can be a multiple factor of $E_k(x)$ over $\mathbb{Q}$ is $x^2-x-1$. Further, the multiplicity of any multiple root of $E_k$ is $2$.
\end{lemma}

For simplicity we prove Theorems \ref{thm:quadpolT+} and \ref{thm:quadpolT-} jointly. We introduce the following notation. Let $\texttt{T}_{a,b}^{k\pm}(x) \in \left\{\texttt{T}_{a,b}^{k+}(x),\texttt{T}_{a,b}^{k-}(x)\right\}$.

\begin{proof}[Proof of Theorems \ref{thm:quadpolT+} and \ref{thm:quadpolT-}]
Let $g \in \mathbb{Q}[x]$ be an arbitrary polynomial with  $\deg g = 2$, and assume that equation \eqref{eq:Tg+} or \eqref{eq:Tg-} holds. We adapt the argument of the proof of Theorem \ref{thm:quadpolS} to our case.
There exist rational numbers $A,B,C$ with $A \neq 0$ such that $g(x) = Ax^2+Bx+C$, and one can rewrite \eqref{eq:Tg+} or \eqref{eq:Tg-} as
\begin{equation} \label{eq:effprf9}
\texttt{T}_{a,b}^{k\pm} \left(x\right) + \nu = A(y+\mu)^2,
\end{equation}
with $\mu=\frac{B}{2A}$ and $\nu=\frac{B^2-4AC}{4A}$. Thus, to bound $x,y$, in view of Lemma \ref{lem:Brindza}, it is sufficient to show that the polynomial on the left hand side of \eqref{eq:effprf9} has at least three roots of odd multiplicity. By \eqref{eq:How2}, we have
\begin{equation} \label{eq:effprf10}
\texttt{T}_{a,b}^{k\pm} \left(x\right) + \nu=\pm\frac{a^k}{2} \left(E_{k} \left(x+ \frac{b}{a}\right) + s\right)
\end{equation}
with $s= E_{k} \left(\frac{b}{a}\right) + \frac{2\nu}{a^k} \in \mathbb{Q}$,
whence, as $k\geq 7$, Lemma \ref{lem:R2} implies the existence of three simple roots for $\texttt{T}_{a,b}^{k\pm} \left(x\right) + \nu$, which completes the proof.
\end{proof}

Similarly, we prove Theorems \ref{thm:linpolT+} and \ref{thm:linpolT-} jointly by using the above notation.

\begin{proof}[Proof of Theorems \ref{thm:linpolT+} and \ref{thm:linpolT-}]
We adapt the argument of the proof of Theorem \ref{thm:linpolS}.
Rewriting equation \eqref{eq:Tling+} or \eqref{eq:Tling-} as
\begin{equation} \label{eq:effprf11}
\texttt{T}_{a,b}^{k\pm} \left(x\right) - d = cy^{\ell},
\end{equation}
in view of Lemma \ref{lem:BBrH}, we can prove that $\ell$ is bounded if we can show that the polynomial $\texttt{T}_{a,b}^{k\pm} \left(x\right) - d$ has two distinct roots for every $d\in \mathbb{Q}$. Let us assume to the contrary. Then, for some $d\in \mathbb{Q}$, we have
\begin{equation} \label{eq:effprf12}
\texttt{T}_{a,b}^{k\pm} \left(x\right) - d = R(Ux+V)^{k}
\end{equation}
with some $R,U,V\in \mathbb{Q}$. However, as $k\geq 7$, this implies that the derivative
\begin{equation} \label{eq:effprf13}
(\texttt{T}_{a,b}^{k\pm} \left(x\right) - d)'= \pm \frac{a^k}{2} E_{k}' \left(x+ \frac{b}{a}\right) = \pm \frac{ka^k}{2} E_{k-1} \left(x+ \frac{b}{a}\right)
\end{equation}
has a rational root of multiplicity at least $6$, which contradicts Lemma \ref{lem:Br2}. Hence $\ell$ is bounded as required.

If $\ell=2$, the effective upper bound on $\max(|x|,|y|)$, in equations \eqref{eq:Tling+} or \eqref{eq:Tling-}, follows directly from Theorems \ref{thm:quadpolT+} and \ref{thm:quadpolT-}, respectively. Let $\ell \geq 3$, and by the first part of the proof, we may assume that $\ell$ is fixed. Further, in view of Theorems \ref{thm:quadpolT+} and \ref{thm:quadpolT-}, we may suppose, without loss of generality, that $\ell$ is an odd prime.
By \eqref{eq:effprf11} and Lemma \ref{lem:Brindza}, it suffices to show that the polynomial on the left hand side of \eqref{eq:effprf11} has at least two roots of multiplicities coprime to $\ell$. Assuming to the contrary, we have
\begin{equation} \label{eq:effprf14}
\texttt{T}_{a,b}^{k\pm} \left(x\right) - d = (Kx+L)(w(x))^{\ell},
\end{equation}
with some $K,L \in \mathbb{Q}$, $w \in \mathbb{Q}[x]$. Taking derivatives in \eqref{eq:effprf14}, by \eqref{eq:effprf13}, we obtain
\begin{equation} \label{eq:effprf15}
\pm \frac{ka^k}{2} E_{k-1} \left(x+ \frac{b}{a}\right) = w(x)^{\ell-1} \left(K w(x) + \ell (Kx+L)w'(x)\right).
\end{equation}
Thus, every root of $w(x)$ is a root of $\pm \frac{ka^k}{2} E_{k-1} \left(x+ \frac{b}{a}\right)$ of multiplicity at least $\ell-1\geq 2$. This immediately contradicts Lemma \ref{lem:Br2} if $\ell \geq 4$.  
In the case $\ell=3$, Lemma \ref{lem:Br2} implies $w(x)=x^2-x-1$. Then, we infer from \eqref{eq:effprf15} that $k=7$ and that $E_{6} \left(x+ \frac{b}{a}\right)$ should have a multiple root, which is a contradiction again by Lemma \ref{lem:Br2}. 
This completes the proof.

\end{proof}

\section*{Acknowledgements}

Research was supported in part by the HUN-REN Hungarian Research Network and by the NKFIH grants ANN130909 and K128088 of the Hungarian National Research, Development and Innovation Office.

\bibliographystyle{amsplain}

\begin{thebibliography}{HD}

\bibitem{AbrSt}
{\sc M.~Abramowitz} and {\sc I.~A.~Stegun}, Handbook of Mathematical Functions, {\em National Bureau of Standards}, 1964.

\bibitem{BA}
{\sc A.~Bazs\'o}, On alternating power sums of arithmetic progressions, {\em Integral Transforms Spec. Funct.\/}, {\bf 24} (2013), 945--949.

\bibitem{BA2}
{\sc A.~Bazs\'o}, Polynomial values of (alternating) power sums, {\em Acta Math. Hungarica\/}, {\bf 146} (2015), 202--219.

\bibitem{BKLP}
{\sc A.~Bazs\'o}, {\sc D.~Kreso}, {\sc F.~Luca} and {\sc {\'A.}~Pint\'er},
  On equal values of power sums of arithmetic progressions, {\em Glas. Mat. Ser. III\/}, {\bf 47} (2012), 253--263. 

\bibitem{BKLPR}
{\sc A.~Bazs\'o}, {\sc D.~Kreso}, {\sc F.~Luca}, {\sc {\'A.}~Pint\'er} and {\sc Cs.~Rakaczki}
  On equal values of products and power sums of consecutive elements in an arithmetic progression, preprint (2022). 


\bibitem{BPS}
{\sc A.~Bazs\'o}, {\sc {\'A.}~Pint\'er} and {\sc H.~M.~Srivastava},
  On a refinement of Faulhaber's Theorem concerning sums of powers of natural numbers, {\em Appl. Math. Letters\/}, {\bf 25} (2012), 486--489.


\bibitem{Ben13}
{\sc M.~A. Bennett}, A superelliptic equation involving alternating sums of powers, {\em Publ. Math. Debrecen\/}, {\bf 79} (2011), 317--324.
 

\bibitem{BGyP}
{\sc M.~A. Bennett}, {\sc K.~Gy{\H{o}}ry} and {\sc {\'A.}~Pint{\'e}r}, On the {D}iophantine equation {$1\sp k+2\sp k+\dots+x\sp k=y\sp n$}, {\em Compos.
  Math.\/}, {\bf 140} (2004), 1417--1431.

\bibitem{BBrH}
{\sc A.~B\'erczes}, {\sc B. Brindza} and {\sc L.~Hajdu}, On the power values of polynomials, {\em Publ. Math. Debrecen\/}, {\bf 53} (1998), 375--381.



\bibitem{Bilu}
{\sc Y.~F. ~Bilu}, Quadratic factors of $f(x) - g(y)$, {\em Acta Arith.\/}, {\bf 90} (1999), 341--355.

\bibitem{BiluTichy}
{\sc Y.~F. ~Bilu} and {\sc R.~F. ~Tichy}, The Diophantine equation $f(x) = g(y)$, {\em Acta Arith.\/}, {\bf 95} (2000), 261--288.

\bibitem{BBKPT}
{\sc Y.~F. Bilu}, {\sc B.~Brindza}, {\sc P. ~Kirschenhofer}, {\sc {\'A.}~Pint\'er} and {\sc R. F. ~Tichy}, Diophantine equations and Bernoulli polynomials (with an Appendix by A. Schinzel), {\em Compositio Math.\/}, {\bf 131} (2002),
  173--188.

\bibitem{Brill}
{\sc J. ~Brillhart}, On the Euler and Bernoulli polynomials, {\em J. Reine Angew. Math.\/}, {\bf 234} (1969),  45--64.

\bibitem{brindza}
{\sc B. Brindza}, On {$S$}-integral solutions of the equation {$y^{m}=f(x)$}, {\em Acta Math. Hungar.\/}, {\bf 44} (1984), 133--139.

\bibitem{Brind}
{\sc B.~Brindza}, On some generalizations of the diophantine equation $1^k + 2^k + \ldots + x^k = y^z$, {\em Acta Arith.\/}, {\bf 44} (1984), 99--107.



\bibitem{Dilch}
{\sc K.~Dilcher}, On a Diophantine equation involving quadratic characters, {\em Compositio Math.\/} {\bf 57} (1986), 383--403.


\bibitem{GyP}
{\sc K.~Gy\H{o}ry} and {\sc {\'A.}~Pint\'er}, On the equation $1^k + 2^k + \ldots + x^k = y^n$, {\em Publ. Math. Debrecen\/}, {\bf 62} (2003), 403--414.

\bibitem{GyTV}
{\sc K.~Gy\H{o}ry}, {\sc R.~Tijdeman} and {\sc M.~Voorhoeve}, On the equation $1^k + 2^k + \ldots + x^k = y^z$, {\em Acta Arith.\/}, {\bf 37} (1980), 234--240.

 
\bibitem{How}
{\sc F.~T.~Howard}, Sums of powers of integers via generating functions, {\em Fibonacci Quart.\/}, {\bf 34} (1996), 244--256.


\bibitem{JPW}
{\sc M.~Jacobson}, {\sc {\'A.}~Pint\'er} and {\sc P.~G.~Walsh},
  A computational approach for solving $y^2 = 1^k + 2^k + \dots + x^k$, {\em Math. Comp.\/}, {\bf 72} (2003), 2099--2110.

\bibitem{L}
{\sc \'E.~Lucas}, Problem 1180, {\em Nouvelles Ann. Math.\/}, {\bf 14} (1875), 336.

\bibitem{KR}
{\sc D.~Kreso} and {\sc Cs.~Rakaczki},
  Diophantine equations with Euler polynomials, {\em Acta Arith.\/}, {\bf 161} (2013), 267--281.

\bibitem{KS1}
{\sc M.~Kulkarni} and {\sc B.~Sury}, Diophantine equations with Bernoulli polynomials, {\em Acta Arith.\/}, {\bf 116} (2005), 25--34.

\bibitem{KS2}
{\sc M.~Kulkarni} and {\sc B.~Sury}, A class of Diophantine equations involving Bernoulli polynomials, {\em Indag. Math. (N.S.)\/}, {\bf 16} (2005), 51--65.
  
\bibitem{KS}
{\sc M.~Kulkarni} and {\sc B.~Sury}, On the Diophantine equation $1+x+\frac{x^2}{2!}+\ldots+\frac{x^n}{n!}=g(y)$., In: Diophantine equations, 121--134, {\em Tata Inst. Fund. Res. Stud. Math.\/}, Mumbai, 2008.


\bibitem{R}
{\sc Cs.~Rakaczki}, On the Diophantine equation $S_m (x) = g(y)$, {\em Publ. Math. Debrecen\/}, {\bf 65} (2004), 439--460.

\bibitem{R2}
{\sc Cs.~Rakaczki}, On the simple zeros of shifted Euler polynomials, {\em Publ. Math. Debrecen\/}, {\bf 79} (2011), 623--636.


\bibitem{schaf} 
{\sc J. J.~Sch{\"a}ffer} The equation $1^p+2^p+3^p+\cdots+n^p=m^q$ {\em Acta Math.\/}, {\bf 95} (1956), 155--189.

\bibitem{Sch}
{\sc A.~Schinzel}, Polynomials with special regard to reducibility, {\em Cambridge University Press}, 2000.

\bibitem{SchT}
{\sc A.~Schinzel} and {\sc R. ~Tijdeman}, On the equation $y^m = P(x)$, {\em Acta Arith.\/}, {\bf 31} (1976), 199--204.

\bibitem{T}
{\sc R. ~Tijdeman}, Applications of the Gel'fond-Baker method to rational number theory, {\em Topics in Number Theory, Proceedings of the Conference at Debrecen 1974}, Colloq. Math. Soc. J\'anos Bolyai {\bf 13}, pp. 399--416, North-Holland, Amsterdam, 1976.


\bibitem{Urb1}
{\sc J.~Urbanowicz}, On the equation $f(1)1^k + f(2)2^k + \ldots + f(x)x^k + R(x) = by^z$, {\em Acta Arith.\/}, {\bf 51} (1988), 349--368.

\bibitem{Urb2}
{\sc J.~Urbanowicz}, On diophantine equations involving sums of powers with quadratic characters as coefficients, I., {\em Compositio Math.\/}, {\bf 92} (1994), 249--271.

\bibitem{Urb3}
{\sc J.~Urbanowicz}, On diophantine equations involving sums of powers with quadratic characters as coefficients, II., {\em Compositio Math.\/}, {\bf 102} (1996), 125--140.


\bibitem{VGyT}
{\sc M.~Voorhoeve}, {\sc K.~Gy\H{o}ry} and {\sc R.~Tijdeman}, On the diophantine equation $1^k + 2^k + \ldots + x^k + R(x) = y^z$, {\em Acta Math.\/}, {\bf 143} (1979), 1-8; Corr. {\bf 159} (1987), 151--152.

\end{thebibliography}

\end{document}